\documentclass[11pt]{article}

\usepackage[utf8]{inputenc}
\usepackage[T1]{fontenc}
\usepackage{lmodern}
\usepackage{amsmath,amssymb,amsthm,mathtools}
\usepackage{mathrsfs}
\usepackage{enumitem}
\usepackage{geometry}
\usepackage{array}
\usepackage{booktabs}
\usepackage[pagebackref]{hyperref}
\usepackage{cite}

\newtheorem{theorem}{Theorem}[section]
\newtheorem{proposition}[theorem]{Proposition}
\newtheorem{lemma}[theorem]{Lemma}
\newtheorem{corollary}[theorem]{Corollary}

\theoremstyle{definition}
\newtheorem{definition}[theorem]{Definition}

\theoremstyle{remark}
\newtheorem{remark}[theorem]{Remark}

\newcommand{\C}{\mathbb C}
\newcommand{\N}{\mathbb N}
\newcommand{\Pj}{\mathbb P}
\newcommand{\OO}{\mathcal O}
\newcommand{\ord}{\operatorname{ord}}
\newcommand{\Hol}{\operatorname{Hol}}
\newcommand{\Spec}{\operatorname{Spec}}
\newcommand{\GL}{\operatorname{GL}}
\newcommand{\PGL}{\operatorname{PGL}}

\hypersetup{colorlinks=true,linkcolor=blue,citecolor=blue,urlcolor=blue}
\title{Pfaffian Geometry and Indicial Singularities of Third-Order Linear ODEs}

\author{V\'ictor Le\'on
\thanks{ILACVN--CICN, Universidade Federal da Integração Latino-Americana,
Parque Tecnológico Itaipu, Foz do Iguaçu--PR, 85867-970, Brazil.
E-mail: \href{mailto:victor.leon@unila.edu.br}{victor.leon@unila.edu.br}}
\and Bruno Sc\'ardua
\thanks{Instituto de Matem\'atica, Universidade Federal do Rio de Janeiro,
CP 68530, Rio de Janeiro--RJ, 21945-970, Brazil.
E-mail: \href{mailto:bruno.scardua@gmail.com}{bruno.scardua@gmail.com}}
}
\date{}

\begin{document}

\maketitle

\begin{abstract}
We study third-order linear differential equations through their associated Pfaffian and projective foliations. In the regular-singular setting, we introduce a Fuchs-adapted Pfaffian model in which the indicial polynomial appears directly in the special-fiber singular scheme, and Frobenius resonance acquires a projective geometric interpretation.
\end{abstract}

\noindent\textbf{Keywords.} third-order linear ODE, Pfaffian foliation, indicial polynomial, singular scheme, regular singularity, Frobenius resonance, adjoint equation, projective Riccati foliation.

\tableofcontents
\section{Introduction and main results}
\label{sec:introduction}

The projectivization of rank-three linear systems and the resulting $\mathbb P^2$-Riccati geometry are classical; for the Riccati-foliation viewpoint see \cite{SantosScardua2010}. The Frobenius and formal-adjoint constructions are classical as well; see \cite{Ince}. The logarithmic-connection interpretation of indicial exponents as residue eigenvalues is recalled, for instance, in \cite{Szabo2008Extension,Ivanics2020}. We shall use these facts as background rather than as novelty claims. Our aim is more specific: to exhibit a Pfaffian representative that connects the companion and Fuchs-adapted pictures explicitly, and then to read indicial multiplicities and Frobenius resonance directly from the singular geometry of that representative.

The motivation comes from the order-two theory developed in \cite{LeonScardua2021}, where an integrable holomorphic $1$-form was used to relate a scalar equation to foliation theory. In order three the companion system lives on $U\times\mathbb C^3$, hence in complex dimension four, and scalar multiplication produces a radial symmetry in the three-dimensional fibers. The natural quotient therefore has fiber $\mathbb P^2$ and codimension two. The corresponding Pfaffian object is a decomposable $2$-form. Moreover, the same contraction principle used below recovers literally the normalized order-two Pfaffian $1$-form, so the construction provides a concrete continuation of the earlier picture.

The paper has three closely related aims.  We first give an explicit Pfaffian realization of the projectivized companion equation and compare it with the classical adjoint construction.  We then pass to the Fuchs-adapted lattice and identify, in companion coordinates, the scheme carried by the special projective fiber.  Finally, we translate the resonant Frobenius compatibility condition into a projective jet-extension problem.  Projectivization, residue eigenlines, monodromy, and the general eigenscheme--Jordan correspondence are used as classical background; the new point is the explicit way in which these structures meet for a scalar equation of order three.

For the monic equation
\begin{equation}
 u'''+a(w)u''+b(w)u'+c(w)u=0,
 \label{eq:intro-monic}
\end{equation}
set
\[
\mathfrak X=\partial_w+y\partial_x+z\partial_y-(az+by+cx)\partial_z,
\qquad
R=x\partial_x+y\partial_y+z\partial_z,
\]
and let \(\nu=dx\wedge dy\wedge dz\wedge dw\).  The decomposable form
\[
\Omega=i_{\mathfrak X}i_R\nu
\]
has kernel \(\langle\mathfrak X,R\rangle\) away from the zero section and therefore realizes the projectivized companion foliation after quotienting by the radial action.  In the chart \(x\ne0\), with \(s=y/x\) and \(t=z/x\), it becomes
\[
\Omega=x^3\eta_1\wedge\eta_2,
\qquad
\eta_1=ds-(t-s^2)\,dw,
\quad
\eta_2=dt+(c+bs+at+st)\,dw.
\]
The adjoint equation gives the same foliation.  If \(\varphi_1,\varphi_2,\varphi_3\) is a fundamental adjoint system and \(H_i\) are the corresponding Lagrange concomitants, then
\[
\boxed{\;i_R(dH_1\wedge dH_2\wedge dH_3)=W^*(w)\,\Omega.\;}
\]
Thus the local projective first-integral maps \([H_1:H_2:H_3]\) have constant \(\PGL(3,\C)\)-transition maps and provide the natural transverse projective structure.

The local geometry at a regular singular point is the main focus.  Write
\[
\beta=w\frac BA,\qquad \gamma=w^2\frac CA,\qquad \delta=w^3\frac DA,
\]
and introduce the Fuchs variables \(X=u\), \(Y=wu'\), \(Z=w^2u''\).  The singular gauge
\[
\Psi(w,x,y,z)=(w,x,wy,w^2z)
\]
relates the ordinary and Fuchs-adapted Pfaffian forms by
\[
\boxed{\Psi^*\widetilde\Omega=w^4\Omega.}
\]
Hence the two forms define the same foliation for \(w\ne0\), while \(\widetilde\Omega\) extends holomorphically across the singular divisor.  In the affine chart \(X\ne0\), with \(S=Y/X\) and \(T=Z/X\), the restriction to the special fiber of the ambient coefficient ideal of the saturated Fuchs-adapted representative is
\[
\boxed{\bigl(T-S(S-1),I(S)\bigr),}
\]
where
\[
I(\lambda)=\lambda(\lambda-1)(\lambda-2)
 +\beta_0\lambda(\lambda-1)+\gamma_0\lambda+\delta_0
\]
is the indicial polynomial.  Thus an indicial root \(\lambda\) of multiplicity \(m\) gives, at
\[
p_\lambda=[1:\lambda:\lambda(\lambda-1)],
\]
a local algebra isomorphic to \(\C\{\xi\}/(\xi^m)\).  This is the scalar companion realization of the general eigenscheme--Jordan correspondence, with the additional feature that the indicial polynomial itself occurs as a defining equation after elimination of the second projective coordinate.  In particular, multiple indicial roots correspond to non-reduced points of this special-fiber scheme and to degenerate tangential projective linearizations.

There is also a direct interpretation of resonance.  Normalized Frobenius solutions correspond to invariant formal sections of the projective system.  If \(\lambda_j-\lambda_i=m>0\) and the recurrence has been solved through degree \(m-1\), the classical compatibility coefficient
\[
\mathcal R_m=\sum_{r=1}^{m}Q_r(\lambda_i+m-r)a_{m-r}
\]
is exactly the obstruction to extending the chosen invariant \((m-1)\)-jet to order \(m\).

For completeness we also include a short Wronskian--Cramer argument showing that three independent formal power-series solutions force the Fuchs bounds
\[
\ord_0P\ge-1,\qquad \ord_0Q\ge-2,\qquad \ord_0R\ge-3.
\]
Together with the convergence theorem of \cite{LeonRodriguezScardua2021}, this closes the analytic question that motivated the comparison with the order-two theory; it is used here as background rather than as the principal novelty.

The paper is organized as follows. Section~\ref{sec:preliminaries} fixes the regular-singular, Pfaffian, projective, and scheme-theoretic conventions used across the two viewpoints. Section~\ref{sec:formal-fuchs} records the direct formal-basis/Fuchs argument. Section~\ref{sec:canonical-foliation} develops the Pfaffian two-form, the adjoint concomitants and the transverse $\mathbb P^2$-projective structure. Section~\ref{sec:projective} constructs the Fuchs-adapted Pfaffian transform, computes the special-fiber intersection of its saturated coefficient zero scheme in terms of the indicial polynomial, and records the resulting Riccati spectrum. Section~\ref{sec:frobenius-sections} establishes the Frobenius--invariant-section correspondence and the resonant jet obstruction. Section~\ref{sec:examples} treats the Bessel and Laguerre models, and Section~\ref{sec:holonomy} discusses projective holonomy and hyperexponential solutions.

\section{Preliminaries and conventions}
\label{sec:preliminaries}

The paper moves between two standard languages---regular-singular scalar ODEs and singular Pfaffian foliations. We collect here the conventions used throughout, so that the arguments can be read from either side without requiring a common specialized vocabulary.

\subsection{Regular-singular and Frobenius terminology}

Consider locally a monic third-order equation
\begin{equation}
 u'''+P(w)u''+Q(w)u'+R(w)u=0.
 \label{eq:prelim-monic}
\end{equation}
We say that $w=0$ is an \emph{ordinary point} when $P,Q,R$ are holomorphic at the origin. We say that $w=0$ is a \emph{regular singular point} (or a Fuchsian singularity) when
\[
 wP(w),\qquad w^2Q(w),\qquad w^3R(w)
\]
are holomorphic at the origin. Equivalently, with $\theta=w\,d/dw$ and
\[
 \beta=wP,\qquad \gamma=w^2Q,\qquad \delta=w^3R,
\]
equation~\eqref{eq:prelim-monic} becomes
\[
 [\theta(\theta-1)(\theta-2)+\beta(w)\theta(\theta-1)+\gamma(w)\theta+\delta(w)]u=0,
\]
where $\beta,\gamma,\delta$ are holomorphic at $0$. Its \emph{indicial polynomial} is
\[
 I(\lambda)=\lambda(\lambda-1)(\lambda-2)+\beta(0)\lambda(\lambda-1)+\gamma(0)\lambda+\delta(0).
\]
A \emph{formal power-series solution} belongs to $\mathbb C[[w]]$. More generally, a normalized formal Frobenius solution of exponent $\lambda$ is an expression
\[
 u=w^\lambda f(w),\qquad f(w)=1+\sum_{n\ge1}a_nw^n\in1+w\mathbb C[[w]].
\]
Here $w^\lambda$ is treated formally when necessary. Substitution first gives $I(\lambda)=0$ and then a recurrence for the coefficients $a_n$. At degree $n$, the new coefficient $a_n$ is multiplied by $I(\lambda+n)$. Hence, if $\lambda+n$ is another indicial root, the recurrence may cease to determine $a_n$ and instead impose a compatibility condition. If two indicial roots satisfy $\lambda_j-\lambda_i=m\in\mathbb N_{>0}$, we call this a \emph{resonance}. At the resonant degree the coefficient multiplying the new unknown may vanish, leaving a compatibility condition; Section~\ref{sec:frobenius-sections} identifies this condition with an invariant-jet extension obstruction. These conventions agree with the standard Frobenius terminology used below; see, for example, \cite{Ince,LeonRodriguezScardua2021}.

\subsection{Pfaffian, projective, and scheme-theoretic terminology}

Holomorphic integrable forms, holonomy, and first integrals belong to the classical local theory of singular holomorphic foliations; see, for instance, \cite{CerveauMattei,MatteiMoussu}.

Let $M$ be a complex manifold. On the open set where a decomposable holomorphic $2$-form $\Omega$ is nonzero and has kernel dimension two, the distribution
\[
 \ker\Omega=\{v\in TM:i_v\Omega=0\}
\]
is a rank-two distribution. In the situations considered here this distribution is involutive and hence, by the holomorphic Frobenius integrability theorem, defines a codimension-two holomorphic foliation on its regular locus. We verify involutivity directly rather than appeal to a general Pfaffian criterion. A local decomposable representative may be written $\Omega=\eta_1\wedge\eta_2$; multiplication by a nowhere-vanishing holomorphic function does not change the underlying foliation.

For the explicit representatives used below, \emph{saturated} (or primitive) means that their coefficient germs have no common non-unit factor. We use the \emph{coefficient zero scheme} of such a fixed representative for the analytic scheme defined by the ideal generated by its coefficients. This terminology is deliberately local: for the Fuchs-adapted extension the scheme is attached to that logarithmic representative, and we do not claim invariance under singular or meromorphic changes of logarithmic lattice. Here a logarithmic lattice means a holomorphic extension across $w=0$ in which the connection has at most a logarithmic pole; different meromorphic choices of such an extension need not have the same residue.

The radial vector field in a three-dimensional fiber is
\[
 R=x\partial_x+y\partial_y+z\partial_z.
\]
Quotienting nonzero fiber vectors by scalar multiplication replaces $\mathbb C^3\setminus\{0\}$ by $\mathbb P^2$. Thus a homogeneous linear system induces a projective, or Riccati, foliation. In an affine chart $X\ne0$ we use $S=Y/X$ and $T=Z/X$. When we say that a point of the projective fiber is singular, we mean that the induced projective vector field vanishes there. The reduced support of the coefficient zero scheme on the special fiber will be precisely this usual singular set; the scheme structure retains multiplicity information that the set alone does not record.

By a \emph{transverse $(\PGL(3,\mathbb C),\mathbb P^2)$-structure} we mean local submersions to $\mathbb P^2$, constant along the leaves, whose changes of transverse coordinates are restrictions of constant projective transformations.

We shall also use one elementary piece of scheme terminology. If $Z$ is a zero-dimensional analytic scheme and $p\in Z$, its \emph{local length} is
\[
 \ell_p(Z):=\dim_{\mathbb C}\mathcal O_{Z,p}.
\]
For example,
\[
 \ell\!\left(\mathbb C\{\xi\}/(\xi^m)\right)=m,
\]
since $1,\xi,\ldots,\xi^{m-1}$ form a complex basis.

Finally, for $A\in M_3(\mathbb C)$, its \emph{eigenscheme} in $\mathbb P^2$ is the scheme-theoretic eigenline locus
\[
 Av\wedge v=0,
\]
equivalently the scheme defined by the $2\times2$ minors of the $3\times2$ matrix with columns $Av$ and $v$. Its reduced support is the usual set of eigenlines, while its non-reduced structure can record Jordan data; see \cite{AboEklundKahlePeterson2016}. This is exactly the notion that occurs on the special fiber of the Fuchs-adapted projective system.

\section{A complete formal basis and the Fuchs condition}
\label{sec:formal-fuchs}

We first dispose of the analytic point that motivated the comparison with the order-two theory.  Consider
\begin{equation}
 u'''+P(w)u''+Q(w)u'+R(w)u=0,
 \label{eq:normalized-third}
\end{equation}
with \(P,Q,R\) meromorphic at the origin.

\begin{proposition}[Complete formal basis and convergence]\label{prop:formal-convergence}
If \eqref{eq:normalized-third} admits three linearly independent formal power-series solutions, then
\[
\ord_0P\ge-1,\qquad \ord_0Q\ge-2,\qquad \ord_0R\ge-3.
\]
Consequently an analytic third-order equation with a complete formal power-series basis is, after cancellation of a common analytic factor and normalization, ordinary or regular singular, and every formal power-series solution is convergent.
\end{proposition}

\begin{proof}
Let \(V\subset\C[[w]]\) be the three-dimensional solution space.  Using the filtration
\(V_m=V\cap w^m\C[[w]]\), whose successive quotients have dimension at most one, choose a basis \(f_1,f_2,f_3\) with strictly increasing orders
\[
m_1<m_2<m_3,\qquad m_j=\ord_0f_j.
\]
Write \(f_j=a_jw^{m_j}+O(w^{m_j+1})\), \(a_j\ne0\), and put
\[
L_r=(f_1^{(r)},f_2^{(r)},f_3^{(r)}),
\qquad W=\det(L_0,L_1,L_2),
\qquad S=m_1+m_2+m_3.
\]
The lowest-order coefficient of \(W\) is
\[
a_1a_2a_3
\det\!\begin{pmatrix}
1&1&1\\
m_1&m_2&m_3\\
m_1(m_1-1)&m_2(m_2-1)&m_3(m_3-1)
\end{pmatrix},
\]
which is the nonzero Vandermonde \(a_1a_2a_3\prod_{i<j}(m_j-m_i)\).  Hence
\[
\ord_0W=S-3.
\]
Since each \(f_j\) satisfies \eqref{eq:normalized-third},
\[
L_3=-RL_0-QL_1-PL_2,
\]
and Cramer's rule gives
\[
P=-\frac{\det(L_0,L_1,L_3)}{W},\qquad
Q=\frac{\det(L_0,L_2,L_3)}{W},\qquad
R=-\frac{\det(L_1,L_2,L_3)}{W}.
\]
For every \(j,r\), \(\ord_0f_j^{(r)}\ge m_j-r\).  Therefore
\[
\ord_0\det(L_{r_1},L_{r_2},L_{r_3})
\ge S-(r_1+r_2+r_3),
\]
so the three numerators above have orders at least \(S-4,S-5,S-6\), respectively.  Division by \(W\) yields the stated Fuchs bounds.  An irregular singular point therefore cannot possess a complete formal power-series basis.

For an analytic equation, cancel a common factor and divide by the leading coefficient on the punctured neighborhood.  The bounds just proved make the origin ordinary or regular singular; the convergence theorem of \cite{LeonRodriguezScardua2021} then implies convergence of every formal power-series solution.  The bounds are sharp, for example for the Euler equation with solutions \(w,w^2,w^3\).
\end{proof}

We now turn to the projective geometry carried by the equation.

\section{The codimension-two Pfaffian foliation}
\label{sec:canonical-foliation}

We now return to the geometric construction that motivates the paper. Throughout this section the equation is the monic equation \eqref{eq:intro-monic}. Its companion field and radial field are
\[
\mathfrak X=\partial_w+y\partial_x+z\partial_y-(az+by+cx)\partial_z,
\qquad
R=x\partial_x+y\partial_y+z\partial_z.
\]
Since the vertical part of $\mathfrak X$ is linear in $(x,y,z)$, one has
\begin{equation}
\boxed{[\mathfrak X,R]=0.}
\label{eq:XR-commute}
\end{equation}
Let
\[
\nu=dx\wedge dy\wedge dz\wedge dw,
\qquad
\Omega=i_{\mathfrak X}i_R\nu.
\]
A direct contraction gives
\begin{equation}
\begin{aligned}
\Omega={}&x\,dy\wedge dz-y\,dx\wedge dz+z\,dx\wedge dy\\
&-(z^2+ayz+by^2+cxy)\,dx\wedge dw\\
&+(yz+axz+bxy+cx^2)\,dy\wedge dw\\
&+(xz-y^2)\,dz\wedge dw.
\end{aligned}
\label{eq:canonical-two-form}
\end{equation}

\begin{proposition}\label{prop:canonical-kernel}
Outside the zero section $\Sigma=\{x=y=z=0\}$,
\[
\boxed{\ker\Omega=\langle\mathfrak X,R\rangle.}
\]
Moreover $\operatorname{Sing}(\Omega)=\Sigma$, and $\Omega$ defines an integrable codimension-two foliation on $U\times(\mathbb C^3\setminus\{0\})$.
\end{proposition}

\begin{proof}
By definition,
\[
i_{\mathfrak X}\Omega=i_{\mathfrak X}i_{\mathfrak X}i_R\nu=0,
\qquad
i_R\Omega=i_Ri_{\mathfrak X}i_R\nu=0.
\]
Thus $\mathfrak X,R\in\ker\Omega$. If $(x,y,z)\ne0$, then $R\ne0$, while $\mathfrak X$ and $R$ are linearly independent because the $\partial_w$-component of $\mathfrak X$ is $1$ and $R$ is vertical. Contraction with the volume form identifies the nonzero bivector $R\wedge\mathfrak X$ with the nonzero decomposable two-form $\Omega$; hence its kernel is exactly $\langle\mathfrak X,R\rangle$. At the zero section $R=0$, so $\Omega=0$. Finally, \eqref{eq:XR-commute} shows that the kernel distribution is involutive.
\end{proof}

\begin{proposition}[Covariance under linear gauges]\label{prop:gauge-covariance}
Let $G:U\to\GL(3,\mathbb C)$ be holomorphic and let
\[
\Phi(w,v)=(w,G(w)v)
\]
be the induced fiberwise linear change of variables. If $\mathfrak X_G=\Phi_*\mathfrak X$, $R_G=\Phi_*R$ and $\nu_G$ is the standard volume form in the new fiber coordinates, then
\[
\boxed{
\Phi^*\bigl(i_{\mathfrak X_G}i_{R_G}\nu_G\bigr)
=(\det G)\,\Omega.}
\]
Hence the Pfaffian representative changes by a nowhere-vanishing factor, and the codimension-two foliation $\ker\Omega$ is intrinsic under holomorphic linear changes of the companion trivialization.
\end{proposition}

\begin{proof}
The radial field is natural under fiberwise linear maps, so $R_G=\Phi_*R$, while by definition $\mathfrak X_G=\Phi_*\mathfrak X$. Moreover
\[
\Phi^*\nu_G=(\det G)\nu;
\]
the terms containing $dw$ in the differentials of $G(w)v$ disappear after wedging with $dw$. Naturality of interior multiplication therefore gives
\[
\Phi^*\bigl(i_{\mathfrak X_G}i_{R_G}\nu_G\bigr)
=i_{\mathfrak X}i_R\Phi^*\nu_G
=(\det G)\Omega.
\]
\end{proof}

\subsection{Radial quotient and the projective Riccati system}

On the chart $x\ne0$ set
\[
s=\frac yx,\qquad t=\frac zx.
\]
Along the companion field,
\[
s'=t-s^2,
\qquad
t'=-c-bs-at-st.
\]
Thus define
\begin{equation}
\eta_1=ds-(t-s^2)\,dw,
\qquad
\eta_2=dt+(c+bs+at+st)\,dw.
\label{eq:ordinary-pfaff}
\end{equation}

\begin{proposition}\label{prop:omega-pfaff}
On $x\ne0$ one has
\[
\boxed{\Omega=x^3\eta_1\wedge\eta_2.}
\]
In particular, the radial quotient of $\ker\Omega$ is the projectivized companion foliation on $U\times\mathbb P^2$.
\end{proposition}

\begin{proof}
Substitute $y=sx$, $z=tx$ into \eqref{eq:canonical-two-form}, with
\[
dy=s\,dx+x\,ds,
\qquad
dz=t\,dx+x\,dt.
\]
After cancellation of the terms containing the radial differential $dx$, the remaining expression is exactly $x^3\eta_1\wedge\eta_2$. Since $s,t,w$ are invariant under the radial action, the Pfaffian system $\langle\eta_1,\eta_2\rangle$ is the quotient system.
\end{proof}

\begin{remark}[Recovery of the order-two form]\label{rem:order-two-recovery}
The contraction construction literally recovers the integrable $1$-form used in the second-order theory. Indeed, for
\[
u''+p(w)u'+q(w)u=0
\]
set
\[
\mathfrak X_2=\partial_w+y\partial_x-(py+qx)\partial_y,
\qquad
R_2=x\partial_x+y\partial_y,
\qquad
\nu_2=dx\wedge dy\wedge dw.
\]
Then
\[
-i_{\mathfrak X_2}i_{R_2}\nu_2
=-y\,dx+x\,dy+(y^2+pxy+qx^2)\,dw,
\]
which is the normalized order-two Pfaffian form of \cite{LeonScardua2021}. Thus the present two-form is not merely analogous to the order-two construction: it is its degree-two counterpart under the same contraction principle.
\end{remark}

For completeness, the Frobenius condition can also be seen directly:
\[
d\eta_1=-\eta_2\wedge dw+2s\,\eta_1\wedge dw,
\]
\[
d\eta_2=(b+t)\eta_1\wedge dw+(a+s)\eta_2\wedge dw.
\]
Hence $d\eta_i$ belongs to the differential ideal generated by $\eta_1,\eta_2$.

\subsection{The adjoint equation and Lagrange concomitants}

The following is the classical Lagrange-adjoint construction; see, for example, \cite{Ince}. Conceptually, the formal adjoint $L^*$ is characterized by a Lagrange identity of the form
\[
\varphi L[u]-uL^*[\varphi]=\frac{d}{dw}\,\mathcal B(u,\varphi),
\]
for a bilinear concomitant $\mathcal B$. Thus, when $L[u]=0$ and $L^*[\varphi]=0$, the concomitant is constant along solutions. For $L=D^3+aD^2+bD+c$, after multiplying the adjoint equation by $-1$, one obtains
\begin{equation}
\varphi'''-a\varphi''+(b-2a')\varphi'+(b'-a''-c)\varphi=0.
\label{eq:adjoint-third}
\end{equation}
For a solution $\varphi$ define its Lagrange concomitant in companion coordinates by
\begin{equation}
H_\varphi=
\bigl(\varphi''-a\varphi'+(b-a')\varphi\bigr)x
+(a\varphi-\varphi')y+\varphi z.
\label{eq:Hphi}
\end{equation}

\begin{lemma}\label{lem:H-first-integral}
If $\varphi$ solves \eqref{eq:adjoint-third}, then
\[
\mathfrak X(H_\varphi)=0,
\qquad
R(H_\varphi)=H_\varphi.
\]
\end{lemma}

\begin{proof}
Write $H_\varphi=A_\varphi x+B_\varphi y+\varphi z$, where
\[
A_\varphi=\varphi''-a\varphi'+(b-a')\varphi,
\qquad
B_\varphi=a\varphi-\varphi'.
\]
Then
\[
\mathfrak X(H_\varphi)=
(A_\varphi'-c\varphi)x+(A_\varphi+B_\varphi'-b\varphi)y
 +(B_\varphi+\varphi'-a\varphi)z.
\]
The last two coefficients vanish identically, while the first is the left-hand side of \eqref{eq:adjoint-third}. The equality $R(H_\varphi)=H_\varphi$ follows from fiberwise homogeneity of degree one.
\end{proof}

On a simply connected open set $V\subset U$, choose a fundamental system $\varphi_1,\varphi_2,\varphi_3$ of \eqref{eq:adjoint-third}, put $H_i=H_{\varphi_i}$, and define
\[
W^*(w)=\det
\begin{pmatrix}
\varphi_1&\varphi_1'&\varphi_1''\\
\varphi_2&\varphi_2'&\varphi_2''\\
\varphi_3&\varphi_3'&\varphi_3''
\end{pmatrix}.
\]

\begin{theorem}[Adjoint--Pfaffian identity]\label{thm:adjoint-pfaffian-full}
One has
\begin{equation}
\boxed{
i_R(dH_1\wedge dH_2\wedge dH_3)=W^*(w)\,\Omega.}
\label{eq:adjoint-pfaffian}
\end{equation}
Hence the two constructions define the same foliation wherever the adjoint system is fundamental.
\end{theorem}

\begin{proof}
Write
\[
H_i=A_ix+B_iy+C_iz,
\]
with
\[
A_i=\varphi_i''-a\varphi_i'+(b-a')\varphi_i,
\quad B_i=a\varphi_i-\varphi_i',
\quad C_i=\varphi_i.
\]
The coefficient matrix $M=(A_i,B_i,C_i)$ factors as
\[
M=
\begin{pmatrix}
\varphi_1&\varphi_1'&\varphi_1''\\
\varphi_2&\varphi_2'&\varphi_2''\\
\varphi_3&\varphi_3'&\varphi_3''
\end{pmatrix}
\begin{pmatrix}
b-a'&a&1\\
-a&-1&0\\
1&0&0
\end{pmatrix}.
\]
The second matrix has determinant $1$, so $\det M=W^*$.

Set $\Theta=dH_1\wedge dH_2\wedge dH_3$. By Lemma~\ref{lem:H-first-integral}, $i_{\mathfrak X}\Theta=0$. Since $\mathfrak X$ is nowhere zero, choose locally a frame whose first vector is $\mathfrak X$. In complex dimension four, a direct basis computation then shows that the space of $3$-forms annihilated by $i_{\mathfrak X}$ is one-dimensional, generated by $i_{\mathfrak X}\nu$. Hence there is a function $f$ such that $\Theta=f\,i_{\mathfrak X}\nu$. Evaluating on $(\partial_x,\partial_y,\partial_z)$ gives
\[
\Theta(\partial_x,\partial_y,\partial_z)=W^*,
\]
whereas
\[
(i_{\mathfrak X}\nu)(\partial_x,\partial_y,\partial_z)
=\nu(\mathfrak X,\partial_x,\partial_y,\partial_z)=-1.
\]
Thus $\Theta=-W^*i_{\mathfrak X}\nu$. Contracting with $R$ and using $i_Ri_{\mathfrak X}=-i_{\mathfrak X}i_R$ yields
\[
i_R\Theta=W^*i_{\mathfrak X}i_R\nu=W^*\Omega.
\]
\end{proof}

\begin{corollary}[Transverse $\mathbb P^2$-projective structure]\label{cor:transverse-P2}
On every simply connected open subset on which a fundamental adjoint system is chosen, the map
\[
\mathcal D(w,[x:y:z])=[H_1:H_2:H_3]
\]
is a holomorphic submersion $U\times\mathbb P^2\to\mathbb P^2$ constant along the leaves of the projectivized companion foliation. Two such maps differ on connected overlaps by a constant element of $\PGL(3,\mathbb C)$.
\end{corollary}

\begin{proof}
For fixed $w$, the vector $(H_1,H_2,H_3)^t$ equals $M(w)(x,y,z)^t$, and $\det M=W^*\ne0$; the nonvanishing follows also from Abel's identity for a fundamental system. Hence the induced fiber map is projective linear and invertible. Lemma~\ref{lem:H-first-integral} shows that it is constant along the companion trajectories, while homogeneity makes it well defined on projective fibers. If a second fundamental adjoint basis is used, it differs from the first by a constant matrix in $\GL(3,\mathbb C)$, and the corresponding projective maps differ by its class in $\PGL(3,\mathbb C)$.
\end{proof}

On $H_3\ne0$ the two affine first integrals are
\[
F_1=\frac{H_1}{H_3},\qquad F_2=\frac{H_2}{H_3},
\]
and a direct expansion gives
\[
\boxed{
i_R(dH_1\wedge dH_2\wedge dH_3)
=H_3^3\,dF_1\wedge dF_2.}
\]
This makes explicit why order three naturally produces two independent projective first integrals rather than a single codimension-one first integral.

\section{The Fuchs-adapted projective Riccati system}
\label{sec:projective}

Projectivizations of rank-three linear systems fit the classical framework of Riccati foliations with projective fiber; compare \cite{SantosScardua2010}. The projectivization itself is not the point here. We keep track instead of the particular logarithmic extension selected by the Fuchs lattice, because this extension carries a saturated Pfaffian coefficient zero scheme whose special-fiber intersection will encode the indicial multiplicities.

Assume that $w=0$ is a regular singular point of an equation
\[A(w)u'''+B(w)u''+C(w)u'+D(w)u=0,\]
with analytic coefficients. Set
\[
\beta(w)=w\frac{B(w)}{A(w)},\qquad
\gamma(w)=w^2\frac{C(w)}{A(w)},\qquad
\delta(w)=w^3\frac{D(w)}{A(w)}.
\]
By the Fuchs condition,
\[
\beta,\gamma,\delta\in\OO_0.
\]
Let
\[
\theta=w\frac{d}{dw}.
\]
Then the equation becomes
\begin{equation}
\left[
\theta(\theta-1)(\theta-2)
+\beta(w)\theta(\theta-1)
+\gamma(w)\theta
+\delta(w)
\right]u=0.
\label{eq:euler-form}
\end{equation}

Introduce
\[
X=u,\qquad Y=wu',\qquad Z=w^2u''.
\]
A direct computation gives
\[
\theta X=Y,
\qquad
\theta Y=Y+Z,
\]
and, using \eqref{eq:euler-form},
\[
\theta Z=-\delta X-\gamma Y+(2-\beta)Z.
\]
Hence we obtain the logarithmic system
\begin{equation}
\boxed{
\theta V=\mathcal A(w)V,
\qquad
\mathcal A(w)=
\begin{pmatrix}
0&1&0\\
0&1&1\\
-\delta(w)&-\gamma(w)&2-\beta(w)
\end{pmatrix}.}
\label{eq:log-system}
\end{equation}
The matrix $\mathcal A(w)$ is holomorphic at the origin.

For $w\neq0$, the projective coordinates $[X:Y:Z]$ are obtained from the ordinary companion coordinates $[u:u':u'']$ by the diagonal transformation
\[
[u:u':u'']\longmapsto[u:wu':w^2u''].
\]
The matrix $\operatorname{diag}(1,w,w^2)$ is invertible away from the origin and degenerates at $w=0$. Thus \eqref{eq:log-system} gives a Fuchs-adapted projective extension across the singular fiber.

On the affine chart $X\neq0$ put
\[
S=\frac{Y}{X},\qquad T=\frac{Z}{X}.
\]
Then
\[
S=w\frac{u'}u,\qquad T=w^2\frac{u''}u,
\]
and the projectivization of \eqref{eq:log-system} is
\begin{equation}
\boxed{
\begin{aligned}
\theta S&=S+T-S^2,\\
\theta T&=2T-ST-\beta T-\gamma S-\delta.
\end{aligned}}
\label{eq:projective-system}
\end{equation}
Equivalently, it is defined by the holomorphic vector field
\[
\mathcal X
=
w\frac{\partial}{\partial w}
+(S+T-S^2)\frac{\partial}{\partial S}
+(2T-ST-\beta T-\gamma S-\delta)\frac{\partial}{\partial T}.
\]
The divisor $E=\{w=0\}$ is invariant.

\subsection{The Fuchs-adapted Pfaffian transform}
\label{subsec:fuchs-pfaffian}

The logarithmic system has a natural total-space vector field
\begin{equation}
\widetilde{\mathfrak X}
=
w\partial_w
+Y\partial_X
+(Y+Z)\partial_Y
+\bigl(-\delta X-\gamma Y+(2-\beta)Z\bigr)\partial_Z,
\label{eq:fuchs-total-field}
\end{equation}
and radial field
\[
\widetilde R=X\partial_X+Y\partial_Y+Z\partial_Z.
\]
Since the vertical part of $\widetilde{\mathfrak X}$ is fiberwise linear,
\[
[\widetilde{\mathfrak X},\widetilde R]=0.
\]
With
\[
\widetilde\nu=dX\wedge dY\wedge dZ\wedge dw,
\qquad
\boxed{\widetilde\Omega=i_{\widetilde{\mathfrak X}}i_{\widetilde R}\widetilde\nu,}
\]
we obtain a holomorphic two-form across $w=0$. Since it is the contraction of a volume form by the decomposable bivector $\widetilde{\mathfrak X}\wedge\widetilde R$, one has $\widetilde\Omega\wedge\widetilde\Omega=0$; away from its zero set, $\ker\widetilde\Omega=\langle\widetilde{\mathfrak X},\widetilde R\rangle$. Because $[\widetilde{\mathfrak X},\widetilde R]=0$, this kernel distribution is integrable.

\begin{proposition}[Fuchsian transform of the Pfaffian form]
\label{prop:fuchs-transform}
On the punctured neighborhood define
\[
\Psi(w,x,y,z)=(w,X=x,Y=wy,Z=w^2z).
\]
Then
\begin{equation}
\boxed{\Psi^*\widetilde\Omega=w^4\Omega.}
\label{eq:fuchs-pullback}
\end{equation}
Hence $\Omega$ and $\widetilde\Omega$ determine the same codimension-two foliation for $w\ne0$, whereas $\widetilde\Omega$ gives a holomorphic extension through the singular divisor.

In the chart $X\ne0$, put
\[
S=\frac YX,\qquad T=\frac ZX,
\]
and set
\[
F=S+T-S^2,
\qquad
G=2T-ST-\beta T-\gamma S-\delta.
\]
Then
\begin{equation}
\boxed{
\widetilde\Omega
=X^3\bigl(w\,dS\wedge dT-G\,dS\wedge dw+F\,dT\wedge dw\bigr).}
\label{eq:fuchs-pfaff-chart}
\end{equation}
\end{proposition}

\begin{proof}
On $w\ne0$ the fields are $\Psi$-related in the precise sense
\[
d\Psi(\mathfrak X)=w^{-1}\widetilde{\mathfrak X}\circ\Psi,
\qquad
d\Psi(R)=\widetilde R\circ\Psi.
\]
Furthermore,
\[
\Psi^*\widetilde\nu
=dx\wedge d(wy)\wedge d(w^2z)\wedge dw
=w^3\nu.
\]
Naturality of contraction therefore gives
\[
\Psi^*\widetilde\Omega
=i_{w\mathfrak X}i_R(w^3\nu)=w^4\Omega.
\]

For the affine formula introduce
\[
\widetilde\eta_1=w\,dS-F\,dw,
\qquad
\widetilde\eta_2=w\,dT-G\,dw.
\]
The ordinary projective coordinates satisfy $S=ws$ and $T=w^2t$, whence
\[
\Psi^*\widetilde\eta_1=w^2\eta_1,
\qquad
\Psi^*\widetilde\eta_2=w^3\eta_2.
\]
Combining this with $\Omega=x^3\eta_1\wedge\eta_2$ and \eqref{eq:fuchs-pullback} yields, for $w\ne0$,
\[
\widetilde\Omega=\frac{X^3}{w}\,\widetilde\eta_1\wedge\widetilde\eta_2.
\]
But
\[
\widetilde\eta_1\wedge\widetilde\eta_2
=w\bigl(w\,dS\wedge dT-G\,dS\wedge dw+F\,dT\wedge dw\bigr),
\]
so the apparent factor $1/w$ cancels and \eqref{eq:fuchs-pfaff-chart} extends holomorphically to $w=0$.
\end{proof}

\begin{lemma}[Saturation of the Fuchs-adapted representative]
\label{lem:fuchs-saturation}
In the affine chart $X\ne0$, write
\[
\omega_F=
 w\,dS\wedge dT-G\,dS\wedge dw+F\,dT\wedge dw,
\]
where
\[
F=S+T-S^2,
\qquad
G=2T-ST-\beta T-\gamma S-\delta.
\]
Then $\omega_F$ is saturated in the elementary local sense that the germs $w,F,G$ have no common non-unit factor in
$\mathbb C\{w,S,T\}$. We denote by
\[
\mathscr S_F:=V(w,F,G)
\]
the coefficient zero scheme of this saturated Fuchs-adapted Pfaffian representative.
\end{lemma}

\begin{proof}
Any common irreducible divisor of $w,F,G$ must divide $w$, hence is associated with
$w$. But
\[
F=S+T-S^2
\]
is not divisible by $w$. Thus $\gcd(w,F,G)=1$, which proves the asserted saturation. The final notation is then simply the definition of the coefficient zero scheme of this representative.
\end{proof}

To identify the scheme on the special fiber, keep the ambient coefficients of the Pfaffian representative and set $w=0$. There
$\widetilde{\mathfrak X}(0,V)=\mathcal A_0V$ and $\widetilde R(0,V)=V$. Since contraction with the fiber volume identifies the bivector $(\mathcal A_0V)\wedge V$ with its three Pl\"ucker coordinates, the three resulting coefficient functions are, up to signs, exactly the $2\times2$ minors of the matrix with columns $\mathcal A_0V$ and $V$. Hence the special-fiber coefficient ideal is the eigenscheme ideal of $\mathcal A_0$ in the sense of \cite{AboEklundKahlePeterson2016}, as an equality of ideals and not merely of zero sets. Notice that this statement concerns evaluation of the ambient coefficient ideal along $E$; the pullback of the two-form itself to $E$ is zero. General eigenscheme theory shows that these schemes encode the numerical data of the Jordan canonical form. The theorem below therefore does not claim the general eigenscheme--Jordan correspondence as new; its point is the explicit realization of that scheme in the distinguished scalar Fuchs companion coordinates.

\begin{theorem}[Indicial roots and the Pfaffian zero scheme]
\label{thm:indicial-pfaffian}
Let $E=\{w=0\}$. In the affine projective chart $X\ne0$, with
\[
S=\frac YX,\qquad T=\frac ZX,
\]
the intersection $\mathscr S_F\cap E$ of the coefficient zero scheme of the saturated Fuchs-adapted Pfaffian representative with the special fiber is defined by
\begin{equation}
\boxed{
T-S(S-1)=0,
\qquad
I(S)=0.}
\label{eq:pfaff-singular-scheme}
\end{equation}
Consequently, its underlying set is
\begin{equation}
\boxed{
 p_\lambda=[1:\lambda:\lambda(\lambda-1)],
 \qquad I(\lambda)=0.}
\label{eq:pfaff-indicial-directions}
\end{equation}
Equivalently, outside the zero section, the zero locus of $\widetilde\Omega$ on $E$ is the union of the residue eigenlines.

Moreover, if $\lambda$ is a root of $I$ of multiplicity $m$, then the local algebra of $\mathscr S_F\cap E$ at $p_\lambda$ is
\begin{equation}
\boxed{
\mathcal O_{\mathscr S_F\cap E,p_\lambda}
\simeq
\frac{\mathbb C\{S-\lambda\}}
{((S-\lambda)^m)}.}
\label{eq:local-singular-algebra}
\end{equation}
In particular, the scheme-theoretic multiplicity of $p_\lambda$ in $\mathscr S_F\cap E$ is exactly the multiplicity of $\lambda$ as a root of the indicial polynomial.
\end{theorem}

\begin{proof}
At $w=0$, the total-space field \eqref{eq:fuchs-total-field} is the vertical linear field $V\mapsto\mathcal A_0V$, while $\widetilde R$ is the radial field $V\mapsto V$. Since contraction with the volume form is injective on bivectors,
\[
\widetilde\Omega(0,V)=0
\iff
(\mathcal A_0V)\wedge V=0.
\]
For $V\ne0$ this is equivalent to $\mathcal A_0V=\lambda V$ for some $\lambda$. The first two rows give
\[
Y=\lambda X,
\qquad
Z=\lambda(\lambda-1)X.
\]
An eigenvector cannot have $X=0$, and the last row is equivalent to $I(\lambda)=0$. This proves the set-theoretic statement in \eqref{eq:pfaff-indicial-directions}.

For the scheme structure, Lemma~\ref{lem:fuchs-saturation} fixes the saturated representative and its coefficient zero scheme $\mathscr S_F$. Formula \eqref{eq:fuchs-pfaff-chart} gives its local ideal through the three coefficients
\[
w,\qquad F=S+T-S^2,\qquad G=2T-ST-\beta T-\gamma S-\delta.
\]
Restricting to $E$ gives the ideal $(F_0,G_0)\subset\mathbb C\{S,T\}$, where
\[
F_0=T-S(S-1),
\]
and
\[
G_0=2T-ST-\beta_0T-\gamma_0S-\delta_0.
\]
Since $\partial F_0/\partial T=1$, the equation $F_0=0$ eliminates $T$ holomorphically. Substituting $T=S(S-1)$ into $G_0$ gives
\begin{align*}
G_0\bigl(S,S(S-1)\bigr)
&=2S(S-1)-S^2(S-1)
   -\beta_0S(S-1)-\gamma_0S-\delta_0\\
&=-\Bigl[S(S-1)(S-2)
+\beta_0S(S-1)+\gamma_0S+\delta_0\Bigr]\\
&=-I(S).
\end{align*}
In fact one has the exact identity
\[
G_0=-I(S)+(2-S-\beta_0)F_0.
\]
Hence
\[
(F_0,G_0)=\bigl(T-S(S-1),I(S)\bigr),
\]
which proves \eqref{eq:pfaff-singular-scheme} as an equality of ideals, not merely set-theoretically.

If $\lambda$ has multiplicity $m$, write
\[
I(S)=(S-\lambda)^mJ(S),\qquad J(\lambda)\ne0.
\]
The germ $J$ is a unit in the local ring at $\lambda$, so elimination of $T$ yields
\[
\frac{\mathbb C\{S-\lambda,T-\lambda(\lambda-1)\}}
{(F_0,G_0)}
\simeq
\frac{\mathbb C\{S-\lambda\}}
{((S-\lambda)^m)}.
\]
This proves \eqref{eq:local-singular-algebra} and the multiplicity statement.
\end{proof}

\begin{remark}[Eigenschemes, gauges, and the Fuchs lattice]
The equations
\[
T=S(S-1),\qquad I(S)=0
\]
are written in the distinguished Fuchs companion coordinates
$(X,Y,Z)=(u,wu',w^2u'')$.  Their scheme-theoretic meaning is the eigenscheme of the residue matrix.  If two logarithmic systems are related by a holomorphic invertible gauge
\[
V=G(w)\widetilde V,\qquad G\in\mathrm{GL}(3,\mathcal O_0),
\]
then
\[
\widetilde{\mathcal A}_0=G(0)^{-1}\mathcal A_0G(0).
\]
Hence the induced projective automorphism identifies their residue eigenschemes, in accordance with the similarity invariance in \cite{AboEklundKahlePeterson2016}; in particular the local scheme lengths are preserved.  We do not claim invariance under singular or meromorphic changes of logarithmic lattice.  Such changes may shift residue eigenvalues by integers and alter the chosen logarithmic extension; compare, for example, the discussion of logarithmic lattices and residues in \cite{Szabo2012Nahm}.  Thus the explicit ideal above belongs to the Fuchs-adapted lattice, whereas its regular-gauge isomorphism class and local lengths are stable under holomorphic invertible changes of frame.
\end{remark}

\begin{remark}[Total length]
Eliminating $T$ gives\[\mathbb C[S,T]/(T-S(S-1),I(S))\simeq\mathbb C[S]/(I(S)).\]Hence the zero-dimensional eigenscheme has total length three, counted with scheme-theoretic multiplicity. This agrees with the general eigenscheme description of a cyclic residue matrix of size three.
\end{remark}

\begin{remark}
The factor $w^4$ in \eqref{eq:fuchs-pullback} has a simple origin: $w^3$ is the Jacobian factor of the singular fiber gauge $(x,y,z)\mapsto(x,wy,w^2z)$, while the remaining factor $w$ comes from replacing $d/dw$ by the Euler field $\theta=w\,d/dw$. Thus the logarithmic Pfaffian form is the Fuchs-adapted holomorphic extension induced by this singular gauge transformation.
\end{remark}

\subsection{Indicial roots and the residue}

Let
\[
\beta_0=\beta(0),\qquad
\gamma_0=\gamma(0),\qquad
\delta_0=\delta(0).
\]
The indicial polynomial is
\begin{equation}
I(\lambda)
=
\lambda(\lambda-1)(\lambda-2)
+\beta_0\lambda(\lambda-1)
+\gamma_0\lambda
+\delta_0.
\label{eq:indicial}
\end{equation}
The residue matrix is
\[
\mathcal A_0=
\begin{pmatrix}
0&1&0\\
0&1&1\\
-\delta_0&-\gamma_0&2-\beta_0
\end{pmatrix}.
\]

\begin{proposition}
\label{prop:indicial-residue}
We have
\[
\boxed{\det(\lambda I-\mathcal A_0)=I(\lambda).}
\]
\end{proposition}

\begin{proof}
Indeed,
\[
\lambda I-\mathcal A_0
=
\begin{pmatrix}
\lambda&-1&0\\
0&\lambda-1&-1\\
\delta_0&\gamma_0&\lambda-2+\beta_0
\end{pmatrix},
\]
and expansion along the first row gives
\begin{align*}
\det(\lambda I-\mathcal A_0)
&=\lambda\left[(\lambda-1)(\lambda-2+\beta_0)+\gamma_0\right]+\delta_0\\
&=I(\lambda).
\end{align*}
\end{proof}

Thus the indicial roots are precisely the eigenvalues of the residue.

\begin{proposition}
\label{prop:projective-singularities}
The singular points of the induced projective vector field on the special fiber $E\simeq\Pj^2$ are exactly
\[
\boxed{
p_\lambda=[1:\lambda:\lambda(\lambda-1)],
\qquad I(\lambda)=0.}
\]
All these points belong to the affine chart $X\neq0$.
\end{proposition}

\begin{proof}
Let $v=(X,Y,Z)^t\neq0$ satisfy
\[
\mathcal A_0v=\lambda v.
\]
The first two rows give
\[
Y=\lambda X,
\qquad
Z=\lambda(\lambda-1)X.
\]
If $X=0$, then $Y=Z=0$, a contradiction. Hence $X\neq0$, and after normalizing $X=1$ we obtain
\[
v_\lambda=(1,\lambda,\lambda(\lambda-1))^t.
\]
The last row is equivalent to $I(\lambda)=0$.
\end{proof}

The projective point $p_\lambda$ will also be denoted in the affine chart by
\[
P_\lambda=(0,\lambda,\lambda(\lambda-1)),
\]
where the first coordinate is the base coordinate $w$.

\subsection{The projective spectrum}

Let $\lambda_1,\lambda_2,\lambda_3$ be the indicial roots, counted with multiplicity.

\begin{theorem}[Projective spectrum]
\label{thm:projective-spectrum}
At $P_{\lambda_i}$ one has
\[
\boxed{
\Spec D\mathcal X(P_{\lambda_i})
=
\{1,\lambda_j-\lambda_i,\lambda_k-\lambda_i\},
}
\qquad \{i,j,k\}=\{1,2,3\}.
\]
\end{theorem}

\begin{proof}
The base coordinate satisfies $\theta w=w$, hence the transverse eigenvalue is $1$. On the special fiber, the vector field is the projectivization of the linear field defined by $\mathcal A_0$. The tangent space to $\mathbb P^2$ at the eigenline $[v_i]$ is naturally $\operatorname{Hom}(\mathbb Cv_i,\mathbb C^3/\mathbb Cv_i)$. The derivative of the projectivized linear field is induced by $\mathcal A_0-\lambda_iI$ on the quotient $\mathbb C^3/\mathbb Cv_i$. Its characteristic polynomial therefore has roots
\[
\lambda_j-\lambda_i,
\qquad
\lambda_k-\lambda_i,
\]
counted with algebraic multiplicity. This formulation does not require $\mathcal A_0$ to be semisimple or the remaining roots to have independent eigenlines.
This proves the assertion.
\end{proof}

A direct computation in the affine coordinates $(S,T)$ gives the same result. The tangential block at $P_\lambda$ is
\[
M_\lambda=
\begin{pmatrix}
1-2\lambda&1\\
-\lambda(\lambda-1)-\gamma_0&2-\lambda-\beta_0
\end{pmatrix}.
\]
If $\lambda=\lambda_i$, then
\[
\operatorname{tr}M_{\lambda_i}
=(\lambda_j-\lambda_i)+(\lambda_k-\lambda_i)
\]
and
\[
\det M_{\lambda_i}
=I'(\lambda_i)
=(\lambda_j-\lambda_i)(\lambda_k-\lambda_i).
\]
In particular, a multiple indicial root gives a zero tangential eigenvalue. By Theorem~\ref{thm:indicial-pfaffian}, this is exactly the case in which the corresponding point of $\mathscr S_F\cap E$ is non-reduced.

\begin{corollary}[Multiple indicial roots]
\label{cor:multiple-indicial-root}
For an indicial root $\lambda$, the following conditions are equivalent:
\begin{enumerate}[label=\rm(\roman*)]
\item $\lambda$ is a multiple root of $I$;
\item the point $p_\lambda$ is non-reduced in $\mathscr S_F\cap E$;
\item the tangential linearization at $P_\lambda$ is degenerate;
\item $0$ is a tangential characteristic exponent at $P_\lambda$.
\end{enumerate}
\end{corollary}

\begin{proof}
By Theorem~\ref{thm:indicial-pfaffian}, the point $p_\lambda$ is non-reduced in $\mathscr S_F\cap E$ exactly
when the multiplicity of $\lambda$ as a root of $I$ is greater than one. On the other
hand,
\[
\det M_\lambda=I'(\lambda).
\]
Hence $\lambda$ is multiple if and only if the tangential block is singular, which is
equivalent to the occurrence of the characteristic exponent $0$.
\end{proof}

\section{Frobenius solutions and invariant sections}
\label{sec:frobenius-sections}

We now relate the usual Frobenius method with the projective system \eqref{eq:projective-system}. Fix $\lambda\in\C$ and regard $w^\lambda\C[[w]]$ as the free rank-one $\C[[w]]$-module generated by the formal symbol $w^\lambda$. We extend the Euler derivation by
\[
\theta(w^\lambda f)=w^\lambda(\lambda f+\theta f).
\]
A normalized formal Frobenius series of exponent $\lambda$ is an expression
\[
u=w^\lambda f(w),
\qquad
f\in1+w\C[[w]].
\]

\begin{definition}
A formal invariant section through $P_\lambda$ is a pair $S,T\in\C[[w]]$ such that
\[
S(0)=\lambda,
\qquad
T(0)=\lambda(\lambda-1),
\]
and \eqref{eq:projective-system} is satisfied formally.
\end{definition}

\begin{theorem}[Frobenius solutions and invariant sections]
\label{thm:frobenius-section}
There is a one-to-one correspondence between normalized formal Frobenius solutions
\[
u=w^\lambda f,
\qquad f\in1+w\C[[w]],
\]
of \eqref{eq:euler-form} and formal invariant sections through $P_\lambda$. The correspondence is
\[
\boxed{
S=\frac{\theta u}{u}
=\lambda+\frac{\theta f}{f},
}
\]
\[
\boxed{
T=\frac{\theta(\theta-1)u}{u}
=S^2-S+\theta S.
}
\]
Conversely,
\[
\boxed{
f=\exp\bigl(\theta^{-1}(S-\lambda)\bigr).}
\]
In particular, the existence of either object implies $I(\lambda)=0$.
\end{theorem}

\begin{proof}
Let $u=w^\lambda f$ be a normalized formal Frobenius solution and put
\[
S=\frac{\theta u}{u}.
\]
Since $\theta u=Su$, we have
\[
\theta^2u=(\theta S+S^2)u,
\]
so
\[
T:=\frac{\theta(\theta-1)u}{u}=\theta S+S^2-S.
\]
This gives the first equation of \eqref{eq:projective-system}. Moreover,
\[
Tu=\theta(\theta-1)u,
\]
and therefore
\[
\frac{\theta(\theta-1)(\theta-2)u}{u}
=\theta T+ST-2T.
\]
Dividing \eqref{eq:euler-form} by $u$ gives the second equation of \eqref{eq:projective-system}.

Conversely, if $(S,T)$ is a formal invariant section, then $S-\lambda\in w\C[[w]]$ and the equation
\[
\theta(\log f)=S-\lambda
\]
has the unique solution with $f(0)=1$ given by
\[
f=\exp\bigl(\theta^{-1}(S-\lambda)\bigr).
\]
Setting $u=w^\lambda f$ recovers $S$ and $T$, and the second projective equation is precisely the scalar equation divided by $u$.
\end{proof}

\subsection{Jets and resonances}

Write
\[
f=1+\sum_{n\geq1}a_nw^n,
\qquad
S=\lambda+\sum_{n\geq1}s_nw^n.
\]
Since
\[
S-\lambda=\frac{\theta f}{f},
\]
one obtains, for every $n\geq1$,
\[
\boxed{
s_n=na_n+\Phi_n(a_1,\ldots,a_{n-1}),
}
\]
where $\Phi_n$ is polynomial. Thus the correspondence between the scalar and projective jets is triangular and formally invertible.

Write
\[
\beta(w)=\sum_{r\geq0}\beta_rw^r,
\qquad
\gamma(w)=\sum_{r\geq0}\gamma_rw^r,
\qquad
\delta(w)=\sum_{r\geq0}\delta_rw^r,
\]
and, for $r\geq1$, put
\[
Q_r(\rho)=\beta_r\rho(\rho-1)+\gamma_r\rho+\delta_r.
\]
If
\[
u=w^\lambda\sum_{n\geq0}a_nw^n,
\qquad a_0=1,
\]
then
\[
\mathscr Lu
=w^\lambda\sum_{n\geq0}E_nw^n,
\]
where
\begin{equation}
\boxed{
E_n
=I(\lambda+n)a_n
+\sum_{r=1}^nQ_r(\lambda+n-r)a_{n-r}.}
\label{eq:frobenius-recurrence}
\end{equation}
The Frobenius recurrence is $E_n=0$.

Let
\[
G=\theta T+ST-2T+\beta T+\gamma S+\delta
\]
be the residual of the second projective equation. By the preceding proof,
\begin{equation}
\boxed{
G=\frac{\mathscr L(w^\lambda f)}{w^\lambda f}.}
\label{eq:residual-identity}
\end{equation}

\begin{proposition}[Resonant jet obstruction]
\label{prop:resonant-obstruction}
Suppose
\[
\lambda_j-\lambda_i=m\in\N_{>0}
\]
and consider the Frobenius recurrence associated with $\lambda_i$. Assume
\[
E_0=\cdots=E_{m-1}=0.
\]
Then the coefficient of $w^m$ in $G$ is
\[
\mathcal R_m
=
\sum_{r=1}^{m}
Q_r(\lambda_i+m-r)a_{m-r}.
\]
Consequently,
\[
\boxed{
\mathcal R_m=0
}
\]
if and only if the chosen invariant $(m-1)$-jet, obtained from a solution of the recurrence equations through degree $m-1$, extends to an invariant $m$-jet.
\end{proposition}

\begin{proof}
Since $f(0)=1$,
\[
\frac1f=1+O(w).
\]
If $E_0=\cdots=E_{m-1}=0$, then
\[
\mathscr Lu=w^{\lambda_i}(E_mw^m+O(w^{m+1})),
\]
and \eqref{eq:residual-identity} gives
\[
G=E_mw^m+O(w^{m+1}).
\]
At the resonant degree,
\[
I(\lambda_i+m)=I(\lambda_j)=0,
\]
so the coefficient $E_m$ contains no $a_m$ term and is precisely $\mathcal R_m$. Since the scalar-to-projective jet correspondence is triangular and formally invertible, vanishing of this coefficient is equivalent to extendability of the chosen projective jet through degree $m$.
\end{proof}

Thus the familiar arithmetic condition
\[
\lambda_j-\lambda_i\in\N
\]
appears simultaneously in Frobenius theory and in the projective dynamics, because $\lambda_j-\lambda_i$ is a characteristic exponent of the projective linearization.

\section{Two classical third-order models}
\label{sec:examples}

We illustrate the projective picture on two equations studied in \cite{LeonRodriguezScardua2021}.

\subsection{The third-order Bessel equation}

Consider the third-order Bessel equation
\begin{equation}
 w^3u'''+3w^2u''+wu'+(w^3-\alpha^3)u=0,
 \qquad \Re(\alpha)\geq0.
 \label{eq:bessel3}
\end{equation}
Here
\[
\beta(w)=3,
\qquad
\gamma(w)=1,
\qquad
\delta(w)=w^3-\alpha^3.
\]
Hence
\[
\mathcal A(w)=
\begin{pmatrix}
0&1&0\\
0&1&1\\
\alpha^3-w^3&-1&-1
\end{pmatrix}
\]
and
\[
I(\lambda)
=\lambda(\lambda-1)(\lambda-2)
+3\lambda(\lambda-1)
+\lambda-\alpha^3
=\lambda^3-\alpha^3.
\]
Let
\[
\omega=e^{2\pi i/3}.
\]
For $\alpha\neq0$ the three indicial roots are
\[
\lambda_1=\alpha,
\qquad
\lambda_2=\alpha\omega,
\qquad
\lambda_3=\alpha\omega^2.
\]
Therefore the special fiber contains the three singular points
\[
p_{\alpha},\qquad
p_{\alpha\omega},\qquad
p_{\alpha\omega^2},
\]
where
\[
p_\lambda=[1:\lambda:\lambda(\lambda-1)].
\]
At $p_\alpha$, for instance, the spectrum is
\[
\left\{1,\alpha(\omega-1),\alpha(\omega^2-1)\right\}.
\]

The case $\alpha=0$, called the third-order Bessel equation of degree zero in \cite{LeonRodriguezScardua2021}, is especially instructive. Then
\[
I(\lambda)=\lambda^3,
\]
so the three indicial roots coalesce at $0$. There is a unique projective singular point
\[
p_0=[1:0:0]
\]
on the special fiber, and both tangential eigenvalues vanish. Indeed, at $\alpha=0$ the residue is
\[
\mathcal A_0=\begin{pmatrix}0&1&0\\0&1&1\\0&-1&-1\end{pmatrix}.
\]
Its characteristic polynomial is $\lambda^3$ and its kernel is $\C(1,0,0)^t$; hence its geometric multiplicity is one and the residue has a single Jordan block of size three. Moreover, Theorem~\ref{thm:indicial-pfaffian} gives
\[
\mathcal O_{\mathscr S_F\cap E,p_0}
\simeq
\frac{\C\{S\}}{(S^3)}.
\]
Thus the three indicial directions coalesce into a non-reduced projective singularity of length three. This is the projective counterpart of the logarithmic hierarchy occurring in the resonant Frobenius description of the degree-zero model.

\subsection{The third-order Laguerre equation}

For the third-order Laguerre equation
\begin{equation}
 w^2u'''+3wu''+(1-w)u'+\alpha u=0.
 \label{eq:laguerre3}
\end{equation}
Multiplying by $w$ gives the Fuchs form
\[
w^3u'''+3w^2u''+w(1-w)u'+\alpha w u=0.
\]
Thus
\[
\beta(w)=3,
\qquad
\gamma(w)=1-w,
\qquad
\delta(w)=\alpha w.
\]
The indicial polynomial is
\[
I(\lambda)
=\lambda(\lambda-1)(\lambda-2)
+3\lambda(\lambda-1)+\lambda
=\lambda^3.
\]
Again the special fiber has a unique projective singularity
\[
p_0=[1:0:0],
\]
with a triple indicial root. Scheme-theoretically,
\[
\mathcal O_{\mathscr S_F\cap E,p_0}
\simeq
\frac{\C\{S\}}{(S^3)},
\]
so this point also has length three on the special singular fiber.

The scalar recurrence is particularly simple. If
\[
u=\sum_{n\geq0}a_nw^n,
\]
then
\[
(n+1)^3a_{n+1}+(\alpha-n)a_n=0,
\]
that is,
\[
\boxed{
a_{n+1}=\frac{n-\alpha}{(n+1)^3}a_n.}
\]
If $\alpha=N\in\N$, the recurrence stops at $n=N$, giving a polynomial solution. This is the third-order analogue of the polynomial phenomenon for the classical Laguerre equation discussed in \cite{LeonRodriguezScardua2021}.

For example, if $\alpha=1$, then
\[
u(w)=1-w
\]
is a solution of \eqref{eq:laguerre3}. The corresponding projective section is
\[
S=w\frac{u'}u=-\frac{w}{1-w},
\qquad
T=w^2\frac{u''}u=0.
\]
It is rational. This example will reappear below in connection with Liouvillian solutions.

\section{Projective holonomy and special solutions}
\label{sec:holonomy}

We finish with two consequences of the projective Riccati interpretation: the projective holonomy group and a simple class of Liouvillian solutions. We use here the standard relation between a linear differential system and its monodromy representation; for background on regular singular systems and monodromy, see, for instance, \cite{Wasow}.

Let
\[
L[u]=A(w)u'''+B(w)u''+C(w)u'+D(w)u
\]
be a third-order equation with polynomial coefficients, and let
\[
\sigma\subset\Pj^1
\]
be the singular set of the associated linear system, including the point at infinity when necessary. Put
\[
B=\Pj^1\setminus\sigma.
\]
Over $B$, we use the ordinary companion system and its projectivization. In this Riccati setting, projective parallel transport along loops in the base is the transverse holonomy of the projectivized foliation; hence the projectivized monodromy representation is precisely its projective holonomy representation. Near a regular singular point this projective system is gauge-equivalent, on the punctured neighborhood, to the Fuchs-adapted model of Section~\ref{sec:projective}. The equation defines a rank-three linear system and hence a monodromy representation
\[
\rho:\pi_1(B,w_0)\longrightarrow\GL(3,\C).
\]
Projectivizing, we obtain
\[
\overline\rho:\pi_1(B,w_0)\longrightarrow\PGL(3,\C).
\]

\begin{definition}
After choosing a fundamental system at the base point, the projective holonomy group of $L$ is
\[
\boxed{
\Hol_{\rm proj}(L)=\operatorname{Im}\overline\rho\subset\PGL(3,\C).}
\]
It is well defined up to conjugacy in $\PGL(3,\C)$.
\end{definition}

\begin{proposition}[Monodromy and transverse projective holonomy]
\label{prop:dual-holonomy}
Let $\Phi$ be a fundamental matrix of the ordinary companion system, normalized at the
base point, and write analytic continuation along $\gamma\in\pi_1(B,w_0)$ as
\[
\Phi^\gamma=\Phi M_\gamma.
\]
Choose the adjoint concomitant covectors so that their coefficient matrix is
$K=\Phi^{-1}$. Then the developing vector
\[
H=(H_1,H_2,H_3)^t
\]
satisfies
\[
H^\gamma=M_\gamma^{-1}H.
\]
Consequently the transverse projective continuation determined by the adjoint
concomitants is represented by $[M_\gamma^{-1}]$. Equivalently, after the usual
row/column identification of the dual system, it is the projective dual
(contragredient) monodromy of the original companion system.
\end{proposition}

\begin{proof}
The rows of $K$ are adjoint covector solutions and satisfy
\[
K'=-KA.
\]
The normalization $K=\Phi^{-1}$ is therefore compatible with the Lagrange pairing.
After continuation,
\[
K^\gamma=(\Phi^\gamma)^{-1}
=(\Phi M_\gamma)^{-1}=M_\gamma^{-1}K.
\]
Since $H=Kv$ in companion coordinates, the same formula gives
$H^\gamma=M_\gamma^{-1}H$, and projectivization yields the assertion.
\end{proof}

This identifies explicitly the transverse projective structure of
Corollary~\ref{cor:transverse-P2} with the projective monodromy data of the scalar
equation (up to the dual/inverse convention).

This is the natural order-three analogue of the global holonomy group considered for second-order equations in \cite{LeonScardua2021}.

\begin{remark}[Elementary consequences of projective monodromy]
If $1\le |\sigma|=r<\infty$, then $\pi_1(\Pj^1\setminus\sigma)$ is free of rank $r-1$, so $\Hol_{\rm proj}(L)$ is generated by at most $r-1$ projective transformations. In particular it is trivial for $r=1$, cyclic for $r=2$, and generated by two elements for $r=3$. Moreover,
\[
\Hol_{\rm proj}(L)=\{1\}
\quad\Longleftrightarrow\quad
\rho(\pi_1(B,w_0))\subset\C^*I_3,
\]
because the kernel of $\GL(3,\C)\to\PGL(3,\C)$ is the subgroup of nonzero scalar matrices.
\end{remark}

We shall finally record one elementary consequence for special solutions. A nonzero solution is called \emph{hyperexponential over $\mathbb C(w)$} when $u'/u\in\mathbb C(w)$; such solutions form a particular class of Liouvillian solutions.

\begin{proposition}[Rational sections and hyperexponential solutions]
\label{prop:rational-hyperexp}
Assume that the coefficients are rational. The Fuchs-adapted projective Riccati system admits a rational invariant section $S,T\in\C(w)$ if and only if the scalar equation admits a nonzero hyperexponential solution
\[
u(w)=\exp\left(\int r(w)\,dw\right),\qquad r(w)\in\C(w).
\]
Indeed, a hyperexponential solution gives
\[
S=w\frac{u'}u=wr,\qquad
T=w^2\frac{u''}u=w^2(r'+r^2),
\]
which is a rational invariant section. Conversely, from a rational invariant section the first projective equation gives
\[
\frac{u'}u=\frac{S(w)}w\in\C(w),\qquad
u=C\exp\left(\int\frac{S(w)}w\,dw\right),
\]
and also $T=\theta S+S^2-S=w^2u''/u$; the second projective equation is then precisely the scalar equation divided by $u$.
\end{proposition}

For the Laguerre model \eqref{eq:laguerre3} with $\alpha=1$, $u=1-w$ gives the rational invariant section
\[
S=-\frac{w}{1-w},\qquad T=0.
\]

A complete characterization of Liouvillian solutions of third-order equations is naturally related to differential Galois theory and is substantially richer than the scalar Riccati analysis available in order two; see, for instance, \cite{vanDerPutSinger}. We do not pursue this problem here. Proposition~\ref{prop:rational-hyperexp} is enough to show that the projective Riccati system still detects an important and computable class of special solutions.

\section{Final remarks}

The constructions above place several classical features of a third-order equation in a single codimension-two picture.  The companion equation gives the Pfaffian form
\[
\Omega=i_{\mathfrak X}i_R(dx\wedge dy\wedge dz\wedge dw),
\]
the adjoint equation gives the same foliation through the Lagrange concomitants, and the Fuchs gauge selects a holomorphic logarithmic extension across a regular singular fiber.  In that extension the residue eigenscheme is represented, in scalar companion coordinates, by
\[
\bigl(T-S(S-1),I(S)\bigr),
\]
so indicial multiplicity becomes scheme length and resonance becomes an invariant-jet obstruction.  This is the point at which the scalar Frobenius and projective Pfaffian descriptions genuinely reinforce one another.

The same contraction principle suggests a higher-order construction with fiber \(\mathbb P^{n-1}\).  We leave that direction aside here in order to keep the paper centered on the first genuinely codimension-two case.



\begin{thebibliography}{99}

\bibitem{CerveauMattei}
D.~Cerveau and J.-F.~Mattei,
\emph{Formes int\'egrables holomorphes singuli\`eres},
Ast\'erisque, No.~97,
Soci\'et\'e Math\'ematique de France,
Paris, 1982.

\bibitem{LeonScardua2021}
V.~Le\'on and B.~Sc\'ardua,
A geometric-analytic study of linear differential equations of order two,
\emph{Electronic Research Archive}
\textbf{29} (2021), no.~2, 2101--2127.
\href{https://doi.org/10.3934/era.2020107}
{doi:10.3934/era.2020107}.

\bibitem{LeonRodriguezScardua2021}
V.~Le\'on, A.~Rodr\'iguez and B.~Sc\'ardua,
Analysis of third order linear analytic differential equations with a regular singularity: Bessel and other classical equations,
\emph{Annals of Mathematical Sciences and Applications}
\textbf{6} (2021), no.~1, 51--83.

\bibitem{AboEklundKahlePeterson2016}
H.~Abo, D.~Eklund, T.~Kahle and C.~Peterson,
Eigenschemes and the Jordan canonical form,
\emph{Linear Algebra Appl.} \textbf{496} (2016), 121--151.
\href{https://doi.org/10.1016/j.laa.2015.12.030}
{doi:10.1016/j.laa.2015.12.030}.

\bibitem{Szabo2008Extension}
Sz.~Szab'o,
The extension of a Fuchsian equation onto the projective line,
\emph{Acta Sci. Math. (Szeged)} \textbf{74} (2008), no.~3--4, 557--564.

\bibitem{Ivanics2020}
P.~Ivanics,
The locus of the representation of logarithmic connections by Fuchsian equations,
\emph{Period. Math. Hungar.} \textbf{81} (2020), 20--45.
\href{https://doi.org/10.1007/s10998-020-00313-6}
{doi:10.1007/s10998-020-00313-6}.

\bibitem{Szabo2012Nahm}
Sz.~Szab'o,
Nahm transform and parabolic minimal Laplace transform,
\emph{J. Geom. Phys.} \textbf{62} (2012), no.~11, 2241--2258.
\href{https://doi.org/10.1016/j.geomphys.2012.07.001}
{doi:10.1016/j.geomphys.2012.07.001}.

\bibitem{Ince}
E.~L.~Ince,
\emph{Ordinary Differential Equations},
Dover Publications, New York, 1956.

\bibitem{MatteiMoussu}
J.-F.~Mattei and R.~Moussu,
Holonomie et int\'egrales premi\`eres,
\emph{Ann. Sci. \'Ecole Norm. Sup. (4)}
\textbf{13} (1980), no.~4, 469--523.
\href{https://doi.org/10.24033/asens.1393}
{doi:10.24033/asens.1393}.

\bibitem{SantosScardua2010}
F.~Santos and B.~Sc\'ardua,
Construction of vector fields and Riccati foliations associated to groups of projective automorphisms,
\emph{Conformal Geometry and Dynamics}
\textbf{14} (2010), 154--166.
\href{https://doi.org/10.1090/S1088-4173-2010-00208-0}
{doi:10.1090/S1088-4173-2010-00208-0}.

\bibitem{vanDerPutSinger}
M.~van der Put and M.~F.~Singer,
\emph{Galois Theory of Linear Differential Equations},
Grundlehren der mathematischen Wissenschaften, Vol.~328,
Springer-Verlag, Berlin, 2003.

\bibitem{Wasow}
W.~Wasow,
\emph{Asymptotic Expansions for Ordinary Differential Equations},
Interscience Publishers, New York, 1965.

\end{thebibliography}
\end{document}